\documentclass{amsart}
\usepackage{amssymb, color}

\title[Critical dynamics on $\operatorname{Per}_n(\lambda)$]{Critical orbits of polynomials with a periodic point of specified multiplier}
\author{Patrick Ingram}
\date{\today}
\address{York University, 4700 Keele St., Toronto, Canada}
\email{pingram@yorku.ca}
\subjclass[2010]{37P30 (Primary) 37P45  (Secondary)}
\thanks{The author would like to thank Laura DeMarco for helpful comments on an earlier version.}

\newcommand{\QQ}{\mathbb{Q}}
\newcommand{\ZZ}{\mathbb{Z}}
\newcommand{\CC}{\mathbb{C}}
\newcommand{\RR}{\mathbb{R}}

\newcommand{\PP}{\mathbb{P}}
\renewcommand{\AA}{\mathbb{A}}

\newcommand{\ord}{\operatorname{ord}}

\renewcommand{\phi}{\varphi}
\renewcommand{\epsilon}{\varepsilon}

\newcommand{\moduli}[1]{\mathsf{M}_{#1}}			
\newcommand{\polym}[1]{\mathsf{P}_{#1}}				
\newcommand{\per}[2]{\operatorname{Per}_{#1}({#2})}

\newtheorem{Theorem}{Theorem}
\newtheorem{Lemma}[Theorem]{Lemma}

\newtheorem{Corollary}[Theorem]{Corollary}

\theoremstyle{remark}
\newtheorem{Remark}{Remark}

\theoremstyle{definition}

\begin{document}
\begin{abstract}
Answering a question posed by Adam Epstein, we show that the collection of conjugacy classes of polynomials admitting a parabolic fixed point and at most one infinite critical orbit is  a set of bounded height in the relevant moduli space. We also apply the methods over function fields to draw conclusions about algebraically parametrized families, and prove an analogous result for quadratic rational maps.
\end{abstract}

\maketitle

\section{Introduction}

The orbits of critical points, and their relation to local behaviour at fixed points, has long been a subject of interest in holomorphic dynamics. The collection of rational functions of a given degree (modulo change of coordinates and ignoring the Latt\`{e}s examples) with all critical orbits finite turns out to be a set of bounded height, a fact conjectured by Silverman~\cite{barbados} and proven by Benedetto, the author, Jones, and Levy \cite{pcfrat} (see also \cite{epstein, pcfpoly, hcrit, alon, js:thurston}).  Silverman's conjecture was motivated in part by Thurston's rigidity result for families of post-critically finite rational functions, and Epstein pointed out to the author that there are other, related rigidity results that might suggest similar arithmetic conjectures. In particular, 
Epstein asked  whether the set of polynomials with a parabolic fixed point and at most one infinite critical orbit is a set of bounded height. In this note, we show that it is.

Let $d\geq 2$, let $\polym{d}$ be the moduli space of polynomials of degree $d$, modulo change of coordinates,  let $h$ be any ample Weil height on $\polym{d}$, and let $\hat{h}_f$ be the canonical height associated to $f$. For $\lambda\in \overline{\QQ}^\times$, let $\per{n}{\lambda}\subseteq \polym{d}$ be the collection of polynomials admitting a point of period $n$ with multiplier $\lambda$.
We remind the reader of the definition of independence below, but on first reading it suffices to note that independent critical points are certainly distinct.

\begin{Theorem}\label{th:main}
For any $d\geq 2$ and $n\geq 1$ there exist constants $\epsilon>0$, $A$, and $B$, such that  for any $\lambda\in \overline{\QQ}^\times$ and any  $f\in \per{n}{\lambda}\subseteq \polym{d}$, either \[h(f)\leq Ah(\lambda)+B\] or else $f$ has independent critical points $c_1$ and $c_2$ such that
\[\min\{\hat{h}_f(c_1), \hat{h}_f(c_2)\}> \epsilon h(f).\]
\end{Theorem}
In light of the relation between the moduli height and the critical height established in \cite{pcfpoly, hcrit}, one may view Theorem~\ref{th:main} as showing that on $\per{n}{\lambda}$, no single critical point accounts for almost all of the critical height of the polynomial $f$. For polynomials in general, of course, this claim is false (for example the unicritical families).

Theorem~\ref{th:main} answers the question asked by Epstein (other variants remain open).

\begin{Corollary}
The set of $f\in \polym{d}$ with a parabolic fixed point and fewer than two independent, infinite critical orbits is a set of bounded height. Consequently, over a given number field $K$ there are only finitely many polynomials of degree $d$ with a parabolic fixed point and fewer than two independent, infinite critical orbits.
\end{Corollary}

Before proceeding, we make a few remarks on the proof. The proof in~\cite{pcfpoly} that the collection of PCF polynomials is a set of bounded height is essentially local, constructing an inequality at each place of a  number field, and concluding the main result simply by summing those inequalities. The analogous result for rational functions~\cite{pcfrat} is similarly local, as is proof of the stronger result that the critical height is commensurate to an ample Weil height on the moduli space~\cite{hcrit} (although in this case the local inequality is not quite a local version of the global inequality, due to some extra terms which vanish when summing over all places).

The main result of this note starts with the same idea as in~\cite{pcfpoly}, namely that in any absolute value, the trivial upper bound on the size of branch points relative to critical points is more-or-less sharp, and so once the critical points are large enough, some branch point is so large as to easily escape to infinity under iteration. Given a critical point $c$ with an infinite orbit, and an absolute value in which $c$ is much smaller than the largest critical point, the argument in fact produces a critical point \emph{other than $c$} which must escape to infinity. This results in a local lower bound on escape rates for critical points excluding $c$, but one that holds only at certain places. The argument is completed not by understanding what happens in the remaining absolute values, but simply by showing that the condition $f\in \per{n}{\lambda}$ ensures that the absolute values in which this argument goes through contribute a positive proportion of $h(f)$.

Just as with the results in \cite{pcfpoly}, the arguments in the present note can be implemented over function fields of irreducible varieties over algebraically closed fields of characteristic 0 or $p>d$, and in this context many of the in-principle-effective constants end up vanishing. 

\begin{Theorem}\label{th:geom}
Let $k$ be an algebraically closed field of characterstic 0 or $p>d$, let $U/k$ be an irreducible quasi-projective variety, fix $\lambda\in k^\times$ and $n\geq 1$, and let \[f:U\to \per{n}{\lambda}\subseteq \polym{d}\] be a non-constant parametrized family of polynomials. Then  $f$ has two generical independent and infinite critical orbits (defined on some extension of the base).
\end{Theorem}

The statement of this result is motivated in part by an ``unlikely intersections'' result of Baker and DeMarco \cite[Theorem~1.2]{bd2}, which shows that in certain families of polynomials over $\CC$ with two independent, infinite critical orbits on the generic fibre, there are only finitely many post-critically finite specializations. Unfortunately, the conditions of \cite[Theorem~1.2]{bd2} (specifically the condition that the critical points are rational on $U=\AA^1$) are such that we are unable to combine that result with Theorem~\ref{th:geom} to conclude finiteness of PCF points on curves in $\per{n}{\lambda}$, for $\lambda\neq 0$, although further results along the lines of \cite{bd2} may allow such an application. We note that the case $n=1$ and $d=3$ of Theorem~\ref{th:geom}, with $k=\CC$, is already apparent in~\cite{bd2}, and was extended by Favre and Gauthier~\cite{fg} to $\per{n}{\lambda}\subseteq \polym{3}$ for arbitrary $n\geq 1$.

Of course, one would like to extend Theorem~\ref{th:main} from polynomials to rational functions, but at the moment we are able to establish this only for quadratic morphisms with marked fixed points. Note that heights of critical points in this setting have already been studied much more deeply by DeMarco, Wang, and Xe~\cite{dwx}, who also obtained lower bounds on the canonical heights of both critical points on $\per{1}{\lambda}$.  The focus there was on equidistribution applications, requiring significantly more detail about the local heights, and not on uniformity in $\lambda$.
\begin{Theorem}\label{th:quad}
There exist constants $\epsilon>0$, $A$, and $B$ such that for any $\lambda\in \overline{\QQ}^\times$ and $f\in \per{1}{\lambda}\subseteq \mathsf{M}_2$, we have either
\[h(f)\leq Ah(\lambda)+B,\]
or else the critical points $c_1,
 c_2$ of $f$ satisfy
\[\min\{\hat{h}_f(c_1), \hat{h}_f(c_2)\}\geq \epsilon h(f).\]
\end{Theorem}

In particular, for $f$ in the normal form
\[f(z)=\frac{\lambda_0 z+ z^2}{\lambda_\infty z+1},\]
we show that both critical points $c$ satisfy
\[\hat{h}_f(c)\geq \frac{1}{32}h(\lambda_\infty)-\frac{25}{32}h(\lambda_0)-\frac{5}{4},\]
 on the hypothesis that $\lambda_0\neq 0$.
The one-parameter family of quadratic morphisms not of this form is treated separately.

Results in this note relate to conjectures made in~\cite{hcrit}. There, we defined the \emph{$k$-depleted critical height} for a rational function $f$ with critical points $c_1, ..., c_{2d-2}$ (listed with multiplicity) by
\[\hat{h}^{(k)}_{\mathrm{crit}}(f)=\min_{\substack{I\subseteq \{1, ..., 2d-2\} \\ |I|=k}}\sum_{i\not\in I} \hat{h}_{f}(c_i),\]
so that
\[0=\hat{h}^{(2d-2)}_{\mathrm{crit}}\leq\cdots \leq \hat{h}^{(1)}_{\mathrm{crit}}\leq  \hat{h}^{(0)}_{\mathrm{crit}}=\hat{h}_{\mathrm{crit}}.\]
Alternatively, one could define $\hat{h}_{\mathrm{crit}}^{(k)}$ by excluding critical points with multiplicity, or even excluding entire dependence classes, but the definition here is more natural as a function on $\mathsf{M}_d$.

In these terms, the main results of~\cite{pcfpoly, hcrit} are that
\[\hat{h}_{\mathrm{crit}}\asymp h_{\polym{d}}\]
away from the flexible Latt\`{e}s examples, while
 Theorems~\ref{th:main} and~\ref{th:quad} prove the conjectured asymptotic
\begin{equation}\label{eq:asymp}\hat{h}_{\mathrm{crit}}^{(1)}\asymp h_{\per{n}{\lambda}}\end{equation}
for polynomials and quadratic rational functions (for $n=1$, in the latter case), with additional uniformity of the implied constants in terms of $\lambda\neq 0$. Note that only one direction of this asymptotic is explicitly treated, but the other follows from $\hat{h}_{\mathrm{crit}}^{(1)}\leq \hat{h}_{\mathrm{crit}}\ll h_{\mathsf{M}_d}$.
We expect~\eqref{eq:asymp} to hold for rational functions in general, excluding the flexible Latt\`{e}s families.

The paper is structured as follows. In Section~\ref{sec:local}, we refine the arguments in~\cite{pcfpoly} to establish lower bounds on escape rates of critical points at certain places. In Section~\ref{sec:global}, we show that the collection of places at which the inequalities from the previous section hold contribute enough to $h(f)$ that using trivial lower bounds at the other places proves Theorem~\ref{th:main}. In Section~\ref{sec:function} we explore implications for families of maps, and in Section~\ref{sec:quad} we treat quadratic rational functions.

\section{Local inequalities}\label{sec:local}

Fix $d\geq 2$.
In this section, we let $K$ stand for a field of characteristic 0 or $p>d$, equipped with some absolute value $|\cdot|$, with associated valuation $v$. We will say that $v$ is \emph{$p$-adic}, for a particular prime $p\in\ZZ$, just in case $0<|p|<1$, and \emph{archimedean} just in case there is an integer $n$ with $|n|>1$. Let $R=\ZZ[\frac{1}{2}, ..., \frac{1}{d}]$, which we may map uniquely to a subring of $K$ given our hypothesis on the characteristic of $K$. We will say that $R$ is $v$-integral if and only if $|x|\leq 1$ for all $x\in R$, noting that this occurs if $|\cdot|$ is neither archimedean nor $p$-adic for any $p\leq d$. Given a point $\mathbf{x}=(x_1, ..., x_n)\in \AA^n(K)$, we set
\[\|\mathbf{x}\|=\max\{|x_1|, ..., |x_n|\}.\]

Anticipating a lemma in the next section which allows us to choose a normal form, we will consider only polynomials of the form
\begin{equation}\label{eq:normalform}
f_{\mathbf{c}}(z)=\frac{1}{d}z^d-\frac{1}{d-1}(c_1+\cdots + c_{d-1})z^{d-1}+\cdots +(-1)^{d-1}c_1c_2\cdots c_{d-1} z,
\end{equation}
for $\mathbf{c}=(c_1, ..., c_{d-1})\in \mathbb{A}^{d-1}$, so that
\[\frac{d}{dz}f_{\mathbf{c}}(z)=(z-c_1)(z-c_2)\cdots(z-c_{d-1}).\] We will fix a $\lambda\in K^\times$ (although any dependence on this value will be tracked explicitly), and we restrict attention to points $\mathbf{c}\in \mathbb{A}^{d-1}$ satisfying \[c_1c_2\cdots c_{d-1}=(-1)^{d-1}\lambda,\] that is, points $\mathbf{c}$ for which the fixed point at $z=0$ for $f_\mathbf{c}$ has multiplier $\lambda$. The map $\mathbf{c}\mapsto f_{\mathbf{c}}$ is a finite map from this restricted domain to $\per{1}{\lambda}\subseteq \polym{d}$, surjective when we extend to the algebraic closure $\overline{K}$.

As usual, we set
\[G_{f_\mathbf{c}}(z)=\lim_{n\to\infty}d^{-n}\log^+|f_{\mathbf{c}}^n(z)|,\]
 where $\log^+x=\log\max\{1, x\}$. The existence of this limit for all $z\in K$ is standard, as is the following lemma.
\begin{Lemma}\label{lem:greensprops}
There exist constants $C_{1}$, $C_{2}$, and $C_{3}$ depending just on $d$ and $K$ such that the following hold.
\begin{enumerate}
\item \label{it:crudebound} For all $z\in K$, \[G_{f_\mathbf{c}}(z)\leq \log^+|z|+\frac{d}{d-1}\log^+\|\mathbf{c}\|+C_1.\]
\item \label{it:escaperegion} If $\log|z|>\log^+\|\mathbf{c}\|+C_2$ then \[-C_3\leq G_{f_\mathbf{c}}(z)-\log^+|z|\leq C_3.\]
\end{enumerate}
Furthermore, we may take $C_1=C_2=C_3=0$ if $R$ is $v$-integral.
\end{Lemma}

\begin{proof}
The argument is standard, coming from an elementary estimate on the difference between $\log^+|f_\mathbf{c}(z)|$ and $d\log^+|z|$; for example see \cite{pcfpoly}. We note here that the constants can easily be made explicit, and we may take $C_1=C_2=C_3=0$ whenever $|1/d|= 1$, and $|1/(d-1)|, ..., |1/2|\leq 1$, which is the case when $R$ is $v$-integral.

%
%

\end{proof}

We now recall the definition of dependence introduced by Baker and DeMarco~\cite{bd2}. Fix a polynomial $f(z)\in K[z]$ with $\deg(f)\geq 2$, and two points $a$, $b$. We say that $a$ and $b$ are \emph{dependendent} if and only if there exist $k$, $m$, $n$ and a non-constant poylnomial $g$ such that \begin{equation}\label{eq:dep}g\circ f^k=f^k\circ g\text{ and }f^n(a)=g\circ f^m(b).\end{equation} Note that if $K\subseteq \CC$, then  a theorem of Ritt~\cite{ritt} implies that in any relation of this form, $g$ is either linear or $f$ and $g$ share an iterate. In particular, dependence is an equivalence relation over $\CC$, and in general we will extend it to be one. Given Ritt's result, we can show that if $a$ and $b$ satisfy a dependence as above, then either $a$ and $b$ are preperiodic or else the quantity $d^{m-n}\deg(g)$ is well-defined (although of course $g$, $n$, and $m$ are not). We will give a simple, self-contained proof of the part of this that we need, which does not assume that we are working over a field of characteristic 0.

Let $\mathcal{G}_{K, f}$ be the set of functions $\psi:K\to\RR$ satisfying \[\psi\circ g = \deg(g)\psi+O_g(1)\] for any polynomial $g(z)\in K[z]$, and $\psi\circ f=\deg(f)\psi$. 
\begin{Lemma}
Let $a, b\in K$ be dependendent under $f(z)\in K[z]$. Then there exists a number $\rho(a, b)\in\RR$ such that for all $\psi\in \mathcal{G}_{K, f}$ we have
\[\psi(a)=\rho(a, b)\psi(b).\]
\end{Lemma}

\begin{proof}
It follows from the definition of $\mathcal{G}_{K, f}$ that whenever $g\circ f^n=f^n\circ g$, we have for any $k\geq 1$
\begin{multline*}
\psi\circ g = \deg(f)^{-{kn}}\psi\circ f^{kn}\circ g = \deg(f)^{-{kn}}\psi\circ g\circ f^{kn} \\ = \deg(f)^{-{kn}}\left(\deg(g)\psi \circ f^{kn}+O_g(1)\right)=\deg(g)\psi+o_g(1),
\end{multline*}
where $o_g(1)\to 0$ as $k\to\infty$. Hence for $g$ commuting with $f$ we have $\psi\circ g = \deg(g)\psi$ for all $\psi\in \mathcal{G}_{K, f}$. The relation $f^n(a)=g\circ f^m(b)$ now implies
\[\psi(a) = d^{-n}\psi(f^n(a)) = d^{-n}\psi\circ g\circ f^m(b)=d^{m-n}\deg(g)\psi(b),\]
and so we may take $\rho(a, b)=d^{m-n}\deg(g)$ for any relation of the form above, which is necessarily well-defined if there is a single $\psi\in \mathcal{G}_{K, f}$ with $\psi(a)\neq 0$. If there is no such $\psi$, then we adopt the convention that $\rho(a, b)=1$.
\end{proof}

We note that the $v$-adic escape-rate function $G_f$ is an element of $\mathcal{G}_{K, f}$, and $\rho(a, b)$ essentially measures how much further along the escape to infinity $a$ is compared to $b$.  The point of the previous lemma is simply that if $a$ and $b$ are dependent, then the ratio $G_f(a)/G_f(b)$, if defined, is in fact independent of the choice of absolute value  (if $K$ admits more than one such choice), a fact that follows from the aforementioned result of Ritt when $K\subseteq \CC$. Note that if $K$ is a number field, then the canonical height $\hat{h}_f$ is also in $\mathcal{G}_{K, f}$, in which case the condition that $\psi(a)=0$ for all $\psi\in \mathcal{G}_{K, f}$ implies that $a$ is preperiodic.

In general, we decompose the set of critical points into dependence-equivalence classes, and declare a representative $c$ to be \emph{$K$-maximal} if and only if either $\psi(c)=0$ for all $\psi\in \mathcal{G}_{K, f}$, or else $\rho(c, b)\geq 1$ for all dependent critical points $b$. Every equivalence class has at least one $K$-maximal representative, because $\rho(b, a)=\rho(a, b)^{-1}$ and $\rho(a, c)=\rho(a, b)\rho(b, c)$. Note that if $K$ is a number field, $c$ being a $K$-maximal representative is equivalent to it having maximal canonical height in its equivalence class.


\begin{Lemma}\label{lem:localindep} For any $d\geq 3$, there exists a constant $C_4$ depending only on $d$ and $K$ such that the following holds.
Suppose that $c_1$ is a $K$-maximal representative of its dependence class, and suppose that \begin{equation}\label{eq:c1small}5\log|c_1|<\log\|\mathbf{c}\|.\end{equation} Then either
\begin{enumerate}
\item $f_\mathbf{c}$ has only one dependency class of critical points, and \begin{equation}\label{eq:maincbound}\log^+\|\mathbf{c}\| \leq C_4,\end{equation} or 
\item  there is a critical point $c_i$, independent of $c_1$, with \begin{equation}\label{eq:mainGbound} G_{f_{\mathbf{c}}}(c_i)\geq \log^+\|\mathbf{c}\|-C_4.\end{equation}
\end{enumerate}
Furthermore, if $R$ is $v$-integral, we may take $C_4=0$.
\end{Lemma}

\begin{proof}
We begin by remarking that if we, in any special case, establish the bound~\eqref{eq:maincbound}, then we have proven the lemma in that case. In these cases~\eqref{eq:mainGbound} holds for all critical points, by the non-negativity of $G_{f_{\mathbf{c}}}$, and the non-positivity of the right-hand-side of~\eqref{eq:mainGbound} given~\eqref{eq:maincbound}. Similarly, as long as we insist that that our choice satisfies $C_4\geq 0$, the claim in the lemma is trivially true when $\|\mathbf{c}\|\leq 1$. So we will suppose throughout that $\|\mathbf{c}\| >1$ which will ensure under~\eqref{eq:c1small} that $\|\mathbf{c}\| = \|c_2, ..., c_{d-1}\|$.

For each $2\leq i\leq d-1$, write
\[f_{\mathbf{c}}(c_i)=P_i(c_2, ..., c_{d-1})+c_1Q_i(c_1, ..., c_{d-1}),\]
where $P_i, Q_i$ are homogeneous forms over $R$ of degree $d$ and $d-1$. It was shown in~\cite{pcfpoly} (see also~\cite[Section~2.2.2]{berteloot}) that the homogeneous forms $f_{\mathbf{c}}(c_1), ..., f_{\mathbf{c}}(c_{d-1})$ have no common non-trivial root in any extension of $R/\mathfrak{m}$, for any maximal ideal $\mathfrak{m}\subseteq R$, and so neither do the forms $P_2, ..., P_{d-1}$ in the variables $c_2, ..., c_{d-1}$ (since any common root corresponds to a common root of the previous collection of forms with $c_1=0$). It follows (from Hilbert's Nullstellensatz) that there exist an integer $e$ and forms $A_{i, j}\in R[c_2, ..., c_{d-1}]$ of degree $e-d$ such that
\[c_i^e=A_{i, 2}P_2+\cdots +A_{i, d-1}P_{d-1}\]
for each $2\leq i\leq d-1$. By the triangle inequality, we have
\[d\log\|c_2, ..., c_{d-1}\|\leq \log\|P_2(c_2, ..., c_{d-1}), ..., P_{d-1}(c_2, ..., c_{d-1})\|+C_5,\]
where we may take $C_5=0$ if  $R$ is $v$-integral.

Now, our hypotheses imply that 
\[\log^+\|\mathbf{c}\|=\log\|\mathbf{c}\|=\log\|c_2, ..., c_{d-1}\|,\] so there exists an $i\geq 2$ with \begin{equation}\label{eq:fminusQ}\log|f_{\mathbf{c}}(c_i)-c_1Q_i(\mathbf{c})|\geq d\log^+ \|\mathbf{c}\|-C_5.\end{equation}
Also, since the coefficients of $Q_i$ are in $R$, we have from the triangle inequality
\begin{equation}\label{eq:c1Q}\log|c_1Q_i(\mathbf{c})|\leq\log|c_1|+ (d-1)\log^+\|\mathbf{c}\| + C_6< \left(d-\frac{4}{5}\right)\log^+\|\mathbf{c}\| + C_6\end{equation}
for some constant $C_6$ which we can take to be 0 if $R$ is $v$-integral.

Now, if we have
\begin{equation}\label{eq:morethanhalf}\log|c_1Q_i(\mathbf{c})|\geq \log|f_{\mathbf{c}}(c_i)-c_1Q_i(\mathbf{c})|-\log^+|2|,\end{equation}
then it follows from~\eqref{eq:fminusQ} and~\eqref{eq:c1Q} that
\[\frac{4}{5}\log^+ \|\mathbf{c}\| \leq C_5+C_6+\log^+|2|.\]
Taking $C_4\geq \frac{5}{4}\left( C_5+C_6+\log^+|2|\right)$
in this case we obtain~\eqref{eq:maincbound}. We have seen that this is sufficient to establish the lemma in this case.

So we may suppose that~\eqref{eq:morethanhalf} fails, and so 
\begin{eqnarray*}
\log|f_\mathbf{c}(c_i)|&\geq& \log|f_{\mathbf{c}}(c_i)-c_1Q_i(\mathbf{c})|-\log^+|2|\\&\geq& d\log^+\|\mathbf{c}\|-C_5-\log^+|2|.
\end{eqnarray*}

First suppose that $\log|f_{\mathbf{c}}(c_i)|\leq \log^+\|\mathbf{c}\|+C_2$, so that we may not apply Lemma~\ref{lem:greensprops}~(\ref{it:escaperegion}) to $z=f_\mathbf{c}(c_i)$. In this case,
\[\log^+\|\mathbf{c}\|\leq \frac{1}{d-1}\left(C_2+C_5+\log^+|2|\right),\] which implies~\eqref{eq:maincbound}, as long as we take $C_4\geq \frac{1}{d-1}(C_2+C_5+\log^+|2|)$, and hence the lemma is proved in this case.

On the other hand, suppose that $\log|f_{\mathbf{c}}(c_i)|> \log^+\|\mathbf{c}\|+C_2$ whereupon, by Lemma~\ref{lem:greensprops}~(\ref{it:escaperegion}),
\begin{eqnarray*}
G_{f_{\mathbf{c}}}(c_i)&=&\frac{1}{d}G_{f_{\mathbf{c}}}(f_{\mathbf{c}}(c_i))\\
&\geq& \frac{1}{d}\log|f_\mathbf{c}(c_i)|-\frac{1}{d}C_3\\
&\geq& \log^+\|\mathbf{c}\|-\frac{1}{d}(C_3+C_5+\log^+|2|).
\end{eqnarray*}
Choosing $C_4\geq \frac{1}{d}(C_3+C_5+\log^+|2|)$, this is at least as strong as the lower bound claimed in~\eqref{eq:mainGbound}, and so if $c_i$ is independent of $c_1$, this completes the proof of the lemma.

Otherwise, suppose that $c_i$ is dependent on $c_1$. Since $c_1$ was assumed $K$-maximal, we have $G_{f_{\mathbf{c}}}(c_1)\geq G_{f_{\mathbf{c}}}(c_i)$. There is a lower bound on $G_{f_{\mathbf{c}}}(c_i)$ above, and we can construct an upper bound on $G_{f_{\mathbf{c}}}(c_1)$. In particular, 
note that the homogeneous form $f_{\mathbf{c}}(c_1)\in R[c_1, ..., c_{d-1}]$, is divisible by $c_1^2$. We thus have from~\eqref{eq:c1small} that
\begin{eqnarray*}\log|f_{\mathbf{c}}(c_1)|&\leq& 2\log |c_1|+(d-2)\log\|\mathbf{c}\|+C_7\\
&\leq& \left(d-\frac{8}{5}\right)\log^+\|\mathbf{c}\|+C_7\end{eqnarray*}
for some constant $C_7$ which we can take to be 0 if $R$ is $v$-integral.

It follows from Lemma~\ref{lem:greensprops}~(\ref{it:crudebound}) that 
\begin{eqnarray*}
G_{f_{\mathbf{c}}}(c_1)&=&\frac{1}{d}G_{f_{\mathbf{c}}}(f_\mathbf{c}(c_1))\\
&\leq &\frac{1}{d}\left(\left(d-\frac{8}{5}\right)\log^+\|\mathbf{c}\|+C_7+\frac{d}{d-1}\log^+\|\mathbf{c}\|+C_1\right)\\
&\leq & \left(1-\frac{3d-8}{5d(d-1)}\right)\log^+\|\mathbf{c}\|+\frac{1}{d}\left(C_1+C_7\right).
\end{eqnarray*}

Combining the upper bound on $G_{f_{\mathbf{c}}}(c_1)$ with the lower bound on $G_{f_{\mathbf{c}}}(c_i)$, we have
\[\log^+\|\mathbf{c}\|\leq \frac{5(d-1)}{3d-8}\left(C_1+C_3+C_5+C_7+\log^+|2|\right).\]
Choosing $C_4$ large enough, this establishes~\eqref{eq:maincbound} and hence proves the lemma in the remaining case.
\end{proof}

\section{The proof of Theorem~\ref{th:main}}\label{sec:global}

We now work over a number field $K$, applying the results of the previous section to the  various standard absolute values on $K$. Quantities from the previous section which depend on the place $v$ now acquire an appropriate subscript.

We begin by explaining why we may freely fix a normal form. Silverman has shown~\cite[p.~103]{barbados} that in general we have for a rational function $f$
\[h_{\mathsf{M}_d}(f)\asymp \min_{g\sim f}h_{\operatorname{Hom}_d}(g),\]
where the minimum is taken over functions conjugate to $f$, and $\operatorname{Hom}_d\subseteq \PP^{2d+1}$ is the space of rational functions of degree $d$ parametrized by their coefficients. A normal form corresponds to a subvariety $U\subseteq \operatorname{Hom}_d$, and so we have
\[h_{\mathsf{M}_d}(f_u)\ll h_{\PP^{2d+1}}(u)\]
for any $u\in U$.  In our case, the normal form $f_{\mathbf{c}}$ corresponds to an embedding $\AA^{d-1}\to \operatorname{Hom}_d$, and one can check directly that $h_{\operatorname{Hom}_d}(f_\mathbf{c})\ll h(\mathbf{c})$, and hence
\[h_{\mathsf{M}_d}(f_\mathbf{c})\leq \alpha h(\mathbf{c})+\beta\]
for some constants $\alpha$ and $\beta$ depending on $d$.

 If we can show that for all $\mathbf{c}$ such that $f_\mathbf{c}\in \per{n}{\lambda}$ we have
\[h(\mathbf{c})\leq Ah(\lambda)+B\]
or else there exist two independent critical points $c_1$ and $c_2$ of $f_\mathbf{c}$ with
\[\hat{h}_{f_\mathbf{c}}(c_1), \hat{h}_{f_\mathbf{c}}(c_2)>\epsilon h(\mathbf{c}),\]
then we will have shown that either
\[h_{\moduli{d}}(f_\mathbf{c})\leq \alpha Ah(\lambda)+(\alpha B+\beta)\]
or else
\[\hat{h}_{f_\mathbf{c}}(c_1), \hat{h}_{f_\mathbf{c}}(c_2)> \epsilon\alpha^{-1}( h_{\moduli{d}}(f_\mathbf{c})- \beta)> \frac{\epsilon}{2 \alpha}h_{\moduli{d}}(f_\mathbf{c}),\]
except where $h_{\moduli{d}}(f_\mathbf{c})\leq 2\beta$. This will prove the result for all conjugacy classes in $\per{n}{\lambda}$ containing a polynomial of the form $f_{\mathbf{c}}$, but every polynomial is conjugate over $\overline{K}$ to one of this form.

We can further simplify the argument by restricting to the case $n=1$.
\begin{Lemma}\label{lem:ntoone}
If Theorem~\ref{th:main} is true with $n=1$, then it is true in full generality.
\end{Lemma}

\begin{proof}
Let $f\in \per{n}{\lambda}$ and suppose that Theorem~\ref{th:main} is known in the case $n=1$. If $P$ is a point of period $n$ and multiplier $\lambda$ for $f$, then it is a fixed point of multiplier $\lambda^n$ for $f^n$.  Since $f^n\in \per{1}{\lambda^n}$, we have either
\begin{equation}\label{eq:fnbound}h(f^n)\leq Ah(\lambda^n)+B,\end{equation}
or else there are independent critical points $c_1, c_2$ of $f^n$ such that
\[\hat{h}_{f^n}(c_1), \hat{h}_{f^n}(c_2)\geq \epsilon h(f^n).\]
Note that $\hat{h}_f=\hat{h}_{f^n}$, that $h(\lambda^n)=nh(\lambda)$, and that
\[h(f)\ll \hat{h}_{\mathrm{crit}}(f)=\frac{1}{n}\hat{h}_{\mathrm{crit}}(f^n)\ll h(f^n),\]
by the main result of~\cite{pcfpoly}. In particular,~\eqref{eq:fnbound} implies $h(f)\leq A' h(\lambda)+B'$, for some constants $A'$ and $B'$ depending on $d$ and $n$.

Now suppose that~\eqref{eq:fnbound} is not satisfied. 
Note that for each critical point $c$ of $f^n$, we have $f'(f^j(c))=0$ for some $0\leq j<n$, and so we have critical points $\zeta_1=f^{j_1}(c_1)$ and $\zeta_2=f^{j_2}(c_2)$ of $f$ satisfying
\[\hat{h}_f(\zeta_i)=d^{j_i}\hat{h}_f(c_1)\geq d^{j_1}\epsilon h(f^n)\gg h(f).\]
Adjusting the constant $B'$ if necessary, we then have either $h(f)\leq A'h(\lambda)+B'$ again, or else
\[\hat{h}_f(\zeta_i)\geq \delta h(f)\]
for some $\delta>0$.

It now remains to check that the independence of $c_1$ and $c_2$ under $f^n$ implies the independence of $\zeta_1$ and $\zeta_2$ under $f$.
Suppose to the contrary that $g\circ f^k=f^k\circ g$ for some $k\geq 1$, and that $f^a(\zeta_1)=g\circ f^b(\zeta_2)$, taking $b\geq n$ without loss of generality. Choose $0\leq r, s<n$ so that $a+j_1+r$ and $b+j_2-s$ are divisible by $n$. Then 
\[(f^n)^{(a+j_1+r)/n}(c_1)=f^r\circ f^a(\zeta_1)=f^r\circ g\circ f^b(\zeta_2)=f^r\circ g\circ f^s \circ (f^n)^{(b+j_2-s)/n}(c_2).\]
Since $f^r\circ g\circ f^s$ commutes with $f^{kn}$, given that $g$ commutes with $f^k$, we have exhibited a dependence between $c_1$ and $c_2$ under $f^n$.
\end{proof}

Now that we know that we may restrict attention to the case $n=1$, and to the normal form~\eqref{eq:normalform}, we outline the strategy of the proof. If $f_\mathbf{c}$ has any infinite critical orbits at all, we will let $c_1$ be the critical point of maximal canonical height, and attempt to bound from below the sum of $\hat{h}_{f_\mathbf{c}}(c_i)$ for $c_i$ independent of $c_1$. We obtain a non-trivial contribution to this quantity from each place at which the hypotheses of Lemma~\ref{lem:localindep} are met, and so the last ingredient is an estimate of how much these places contribute to the weighted sum defining $h(\mathbf{c})$.

\begin{Lemma}\label{lem:Sweight}
Suppose that $\prod c_i=(-1)^{d-1}\lambda\neq 0$, and let
\begin{equation}\label{eq:Sdef}S=\left\{v\in M_K:5\log|c_1|_v<\log\|\mathbf{c}\|_v\right\}.\end{equation}
Then 
\[\sum_{v\in S}\frac{[K_v:\QQ_v]}{[K:\QQ]}\log^+\|\mathbf{c}\|_v\geq\frac{1}{5d-9}h(\mathbf{c})-\frac{5d-4}{5d-9}h(\lambda).\]
\end{Lemma}

\begin{proof}
To ease notation, we set $n_v=[K_v:\QQ_v]/[K:\QQ]$.
Note that the relation $\prod_{i\geq 1} c_i=\pm\lambda$ gives us both
\[\log\|\mathbf{c}\|_v\leq \log^+\|\mathbf{c}\|_v\leq \log\|\mathbf{c}\|_v+\frac{1}{d-1}\log^+|\lambda|_v\]
and
\[|c_1|_v^{-1}\leq |\lambda^{-1}|_v\|\mathbf{c}\|_v^{d-2}.\]
We apply these and the product formula to obtain
\begin{eqnarray*}
\sum_{v\not\in S}n_v\log^+\|\mathbf{c}\|_v&\leq & \sum_{v\not\in S}n_v\left(\log\|\mathbf{c}\|_v+\frac{1}{d-1}\log^+|\lambda|_v\right)\\ 
&\leq &\sum_{v\not\in S}n_v5\log|c_1|_v+\sum_{v\not\in S}\frac{1}{d-1}n_v\log^+|\lambda|_v\\
&\leq &\sum_{v\in S}n_v5\log|c_1|^{-1}_v + \frac{1}{d-1}h(\lambda)\\
&\leq &\sum_{v\in S}n_v5\left(\log^+|\lambda^{-1}|_v+(d-2)\log^+\|\mathbf{c}\|_v\right)\\&&+\frac{1}{d-1}h(\lambda)\\
&\leq & 5(d-2)h(\mathbf{c})-\sum_{v\not\in S}5(d-2)n_v \log^+\|\mathbf{c}\|_v \\&&+ \left(5+\frac{1}{d-1}\right) h(\lambda),
\end{eqnarray*}
and so
\[\sum_{v\not\in S}n_v\log^+\|\mathbf{c}\|_v\leq \frac{5(d-2)}{1+5(d-2)}h(\mathbf{c})+\frac{5d-4}{1+5(d-2)}h(\lambda),\]
whereupon
\begin{equation}\label{eq:cbound}\sum_{v\in S}n_v\log^+\|\mathbf{c}\|_v\geq \frac{1}{5d-9}h(\mathbf{c})-\frac{5d-4}{5d-9}h(\lambda).\end{equation}
\end{proof}


\begin{proof}[Proof of Theorem~\ref{th:main}]
As noted above, it will suffice to prove the result with $n=1$ for polynomials of the form $f_{\mathbf{c}}$, taking $h(f_{\mathbf{c}})=h(\mathbf{c})$. In the case $d=2$, we have $\mathbf{c}=c_1=-\lambda$ on $\per{1}{\lambda}$, and so the result follows immediately. We will assume from now on that $d\geq 3$.

If $f_\mathbf{c}$ is post-critically finite (PCF), then we have from \cite{pcfpoly} a bound on $h(\mathbf{c})$, so the conclusion of the theorem holds. Suppose that $f_{\mathbf{c}}$ is not PCF, and without loss of generality suppose that $c_1$ has maximal canonical height (which is positive). It follows  that $c_1$ is a maximal representative of its dependency class, $D$, and in particular that $G_{f_{\mathbf{c}}, v}(c_1)\geq G_{f_{\mathbf{c}}, v}(c_i)$ for any $c_i$ dependent on $c_1$, and any place $v\in M_K$.

Note that we might have $D=\{c_1, ..., c_{d-1}\}$, in which case we must bound $h(\mathbf{c})$. Let $S$ be the set of places defined in~\eqref{eq:Sdef}, and note that by Lemma~\ref{lem:localindep} we have, for each $v\in S$,
\[\log^+\|\mathbf{c}\|_v \leq C_{4, v}.\]
Applying Lemma~\ref{lem:Sweight}, we have a constant $C_{21}$ such that
\begin{eqnarray*}
\frac{1}{5d-9} h(\mathbf{c})-\frac{5d-4}{5d-9}h(\lambda) &\leq & \sum_{v\in S}n_v\log^+\|\mathbf{c}\|_v\\
&\leq & \sum_{v\in S}n_v C_{4, v},
\end{eqnarray*}
and hence
\[h(\mathbf{c})\leq (5d-4)h(\lambda)+(5d-9)\sum_{v\in M_K}n_v C_{4, v}.\]

Now, if at least some critical point is independent of $c_1$, Lemma~\ref{lem:localindep} furnishes at each place $v\in S$ a critical point $c_{i_v}\not\in D$ with
\[G_{f_{\mathbf{c}}, v}(c_{i_v})\geq \log^+\|\mathbf{c}\|_v-C_{4, v}.\]
By the non-negativity of $G_v$, we have
\[\sum_{c_i\not\in D}G_{f_{\mathbf{c}}, v}(c_{i_v})\geq \log^+\|\mathbf{c}\|-C_{4, v}\]
at every place $v\in S$. Again using the non-negativity of $G_{f_{\mathbf{c}}, v}$, and Lemma~\ref{lem:Sweight}, we have
\begin{eqnarray*}
\frac{1}{5d-9} h(\mathbf{c})&\leq& \sum_{v\in S} n_v\log^+\|\mathbf{c}\|_v + \left(\frac{5d-4}{5d-9}\right)h(\lambda)\\
&\leq & \sum_{v\in S} n_v\sum_{c_i\not\in D} n_vG_{f_{\mathbf{c}}, v}(c_i) + \sum_{v\in S} n_v C_{4, v}+ \left(\frac{5d-4}{5d-9}\right)h(\lambda) \\
&\leq & \sum_{c_i\not\in D}\sum_{v\in M_K} n_v G_{f_{\mathbf{c}}, v}(c_i)+ \sum_{v\in M_K} n_v C_{4, v}+ \left(\frac{5d-4}{5d-9}\right)h(\lambda)\\
&=& \sum_{c_i\not\in D}\hat{h}_{f_{\mathbf{c}}}(c_i)+ \sum_{v\in M_K} n_v C_{4, v}+ \left(\frac{5d-4}{5d-9}\right)h(\lambda)\\
&\leq& (d-2)\max_{c_i\not\in D}\hat{h}_{f_{\mathbf{c}}}(c_i)+ \sum_{v\in M_K} n_v C_{4, v}+ \left(\frac{5d-4}{5d-9}\right)h(\lambda).
\end{eqnarray*}
So for some $c_i$ independent of $c_1$, we have 
\begin{eqnarray*}
\hat{h}_{f_\mathbf{c}}(c_i)&\geq&\frac{1}{(d-2)(5d-9)} h(\mathbf{c})- \left(\frac{5d-4}{(d-2)(5d-9)}\right)h(\lambda) \\&&- \left(\frac{1}{d-2}\right) \sum_{v\in M_K} n_v C_{4, v}\\
&>&\frac{1}{2(d-2)(5d-9)} h(\mathbf{c})
\end{eqnarray*}
unless \begin{equation}\label{eq:hboundeq}h(\mathbf{c})\leq (10d-8)h(\lambda)+ 2(5d-9) \sum_{v\in M_K} n_v C_{4, v}.\end{equation}
In other words, we have either~\eqref{eq:hboundeq} or else $f_{\mathbf{c}}$ has independent critical points $c_1$ and $c_2$ satisfying
\[\min\{\hat{h}_{f_{\mathbf{c}}}(c_1), \hat{h}_{f_{\mathbf{c}}}(c_2)\}>\left(\frac{1}{2(d-2)(5d-9)}\right)h(\mathbf{c}),\]
which is what we set out to prove.

As noted, this proves Theorem~\ref{th:main} for conjugacy classes in $\per{1}{\lambda}$ containing a polynomial of the form $f_{\mathbf{c}}$, with $\mathbf{c}\in K^{d-1}$. But the bounds are independent of $K$, and so hold over any finite extension of $K$, and hence over $\overline{K}$. As noted, every conjugacy class in $\per{1}{\lambda}\subseteq \polym{d}$ is the conjugacy class of some $f_{\mathbf{c}}$ over $\overline{K}$.
\end{proof}

\section{Algebraic families in $\per{n}{\lambda}$}\label{sec:function}

Let $k$ be an algebraically closed field of characteristic $0$ or $p>d$, and let $U/k$ be an irreducible quasi-projective variety. Given a set of absolute values $M$ on the function field $K=k(U)$ satisfying the product formula, that is $\prod_{v\in M}|x|_v=1$ for all $x\neq 0$, we define as usual
\[h_{\PP^N, M}([x_0:\cdots :x_N])=\sum_{v\in M}\log^+\|x_0, ..., x_N\|_v.\]
The following standard result asserts that we may always choose such a set of absolute values so that the points of height zero are exactly those defined over the constant field.
\begin{Lemma}\label{lem:ffheights}
Let $k$ be algebraically closed, and let $U/k$ be an irreducible quasi-projective variety. There exists a set $M_K$ of non-archimedean, non-$p$-adic absolute values on $K=k(U)$ satisfying the product formula.
Futhermore, there is a canonical extension of the places in $M_K$ to any finite extension $L/K$ such that $h$ is well-defined on $\overline{K}$, and we have
\[\{P\in \PP^N(\overline{K}):h_{\PP^N, M}(P)=0 \}=\PP^N(k).\]
\end{Lemma}

\begin{proof}
See \cite[Section~1.4 and Example~2.4.11]{bg}, but we remind the reader here that if $U$ is normal and projective, then the absolute values in $M_K$ correspond to prime divisors $Z$ on $X$. We set
\[|x|_Z=e^{-\ord_{Z}(x)\deg(Z)},\]
where $\deg(Z)$ is the degree of $Z$ relative to some chosen ample class on $X$. It follows that $h([x:y])$ is the degree of the pole divisor of $x/y$ relative to $Z$, and since $U$ is normal and projective, only constants have trivial pole divisors. 

In general, $K$ is $k$-isomorphic to the function field of some normal, projective variety, so it suffices to consider that case, although the abundance of non-isomorphic projective normalizations of $U$ suggests correctly that $M_K$ is not itself canonical.
\end{proof}

We note that if $M$ is a absolute values as furnished by Lemma~\ref{lem:ffheights}, then  $|x|_v=1$ for any $x\in k^\times$. In particular, in the terminology of Section~\ref{sec:local}, $R$ is $v$-integral for every $v\in M$.

\begin{Lemma}\label{lem:cink}
Let $\lambda\in k^\times\subseteq K^\times$, and let $\mathbf{c}\in \AA^{d-1}(K)$ satisfy $c_1c_2\cdots c_{d-1}=(-1)^{d-1}\lambda$. Then either $f_{\mathbf{c}}$ has two independent, infinite critical orbits, or else $\mathbf{c}\in \AA^{d-1}(k)$. 
\end{Lemma}

\begin{proof}
With $h$ a height relative to a set $M$ of places as provided by Lemma~\ref{lem:ffheights}, the condition $\mathbf{c}\in \AA^{d-1}(k)$ is equivalent to $h(\mathbf{c})=0$,  and so we will work in terms of heights. Note that our assumptions also imply that $|\lambda|_v=1$ for all $v\in M$.

Let $S\subseteq M$ be the set of places $v$ witnessing $5\log|c_1|_v<\log\|\mathbf{c}\|_v$. If there exists a place $v\in S$ with $\log^+\|\mathbf{c}\|_v\neq 0$,  we may apply Lemma~\ref{lem:localindep} to conclude that $f_{\mathbf{c}}$ has a critical point $c_i$ independent of $c_1$ satisfying $G_{f_{\mathbf{c}}, v}(c_i)\geq \log^+\|\mathbf{c}\|_v> 0$, and so a critical point independent of $c_1$ with an infinite orbit.
 
On the other hand, suppose that we have $\log^+\|\mathbf{c}\|_v=0$ for each $v\in S$. Our hypothesis $\lambda\in k^\times$ ensures that $\log\|\mathbf{c}\|_v=\log^+\|\mathbf{c}\|_v$ and $c_1\neq 0$, so by the product formula we have
\begin{eqnarray*}
h(\mathbf{c})&=&\sum_{v\in M}\log^+\|\mathbf{c}\|_v\\&=&\sum_{v\not\in S}\log^+\|\mathbf{c}\|_v\\
&\leq &5\sum_{v\not\in S}  \log|c_1|_v\\
&=&5\sum_{v\in S}\log|c_1^{-1}|_v\\
&=&5\sum_{v\in S}\log|c_2\cdots c_{d-1}|_v\\
&\leq &5(d-2)\sum_{v\in S}\log^+\|\mathbf{c}\|_v=0.
\end{eqnarray*}

So if $\mathbf{c}\not\in\AA^{d-1}(k)$, then $f_{\mathbf{c}}$ has at least one infinite critical orbit. Re-arranging the indices so that this critical point is $c_1$, we may run the argument through again to find some $c_i$ independent of $c_1$ which also has an infinite forward orbit.
\end{proof}

Recall that a polynomial $f$ defined over a function field $K$ with algebraically closed constant field $k$ is \emph{isotrivial} if and only if it is conjugate over some extension of $K$ to a polynomial with constant coefficients.

\begin{proof}[Proof of Theorem~\ref{th:geom}]
As in the case of number fields, we will first show that it suffices to treat the case $n=1$. Let $K=k(U)$, and let $f\in \per{n}{\lambda}$ be the generic fibre of the family, with $\lambda\in k^\times$. Since $f^n\in \per{1}{\lambda^n}$, we see that either $f^n$ isotrivial (defined over $k$ after a change of variables), or else $f^n$ has two independent infinite critical orbits. In the latter case, one concludes as in the proof of Lemma~\ref{lem:ntoone} that $f$ does as well. But if $f^n$ is isotrivial, then so is  $f$ (see~\cite{benedetto}; the hypothesis therein that $\dim(U)=1$ is superfluous, e.g., see~\cite{baker}).

Now suppose that $f\in \per{1}{\lambda}$ over $K$, with $\lambda\in k^\times$, and suppose that $f$ does not have two independent, infinite critical orbits.  Over some extension of $K$,  $f$ is conjugate to $f_\mathbf{c}$ with $c_1c_2\cdots c_{d-1}=(-1)^{d-1}\lambda$, and hence by Lemma~\ref{lem:cink} we have $\mathbf{c}\in \AA^{d-1}(k)$. Since $f_{\mathbf{c}}$ has constant coefficients and is conjugate to $f$, $f$ is isotrivial.
\end{proof}

\begin{Remark}
Note that while the condition $\hat{h}_{\mathrm{crit}}(f)=0$ over a number field precisely identifies PCF polynomials, in the function field context it precisely identifies those that are isotrivial, that is, conjugate over some extension of $K$ to a polynomial with constant coefficients.

In one direction this is relatively straightforward. If $f$ is defined over the constant field, then $\hat{h}_f(z)=0$ for all $z\in k\subseteq K$, simply because $f(k)\subseteq k$ and $h(z)=0$ for all $z\in k$. Since a $k$-rational polynomial has $k$-rational critical points, it follows that $\hat{h}_{\mathrm{crit}}(f)=0$ for $f$ defined over $k$, and since $\hat{h}_{\mathrm{crit}}$ is well-defined on conjugacy classes, the same is true for any polynomial conjugate to one with constant coefficients.

On the other hand, if $\hat{h}_{\mathrm{crit}}(f)=0$ then $f$ is conjugate to a map of the form $f_{\mathbf{c}}$ with $\mathbf{c}\in \AA^{d-1}(k)$, by the proof of Theorem~\ref{th:geom}. Alternatively, over $k=\CC$, we could appeal to the compactness of the connectedness locus, or Thurston's rigidity theorem to conclude that any PCF family of polynomials is isotrivial.

Similarly, the argument that proves Theorem~\ref{th:geom} shows that for $\lambda\in k^{\times}$ and $f\in \per{n}{\lambda}\subseteq \polym{d}$, we have $f$ isotrivial if and only if $\hat{h}_{\mathrm{crit}}^{(1)}(f)=0$.

This property is a natural one for heights over functions fields. Indeed, suppose that $L$ is an ample divisor on $\mathsf{M}_d$ and $\psi:\mathsf{M}_d\to \PP^N$ is an embedding relative to which $nL$ corresponds to the hyperplane $H\subseteq \PP^N$ at infinity, for some $N, n\geq 1$. Then define $h_{\mathsf{M}_d, L}=\frac{1}{n}\psi^* h_{\PP^N, H}$. If $L$ is $k$-rational, and we choose $\psi$ to be as well, then $h_{\mathsf{M}_d, L}(f)=0$ if and only if $f\in \psi^{-1}(\PP^N(k))=\mathsf{M}_d(k)$. In other words, for ample Weil heights on $\mathsf{M}_d$ constructed in this manner, $h_{\mathsf{M}_d, L}(f)=0$ if and only if $f$ is isotrivial.
\end{Remark}

\section{Quadratic morphisms}\label{sec:quad}

In this section we treat the case of rational functions of degree 2, with the aim of proving Theorem~\ref{th:quad}. We work over a number field $K$  with the usual set of places $M_K$. 

 From now on, let
\[f_{\lambda_0, \lambda_\infty}(z)=\frac{\lambda_0 z+z^2}{\lambda_\infty z + 1}.\] Over $\overline{K}$, every quadratic endomorphism of $\mathbb{P}^1$ is conjugate either to one of this form, as shown by Milnor~\cite{milnor} and Silverman~\cite{md} (see also \cite[Section~4.2]{ads}), or to a member of a one-parameter family treated separately below.
We will think of $\lambda_0$ as being fixed, but explicit dependence on this value will be tracked under the hypothesis only that $\lambda_0\neq 0$. The following result is enough to establish Theorem~\ref{th:quad}, modulo the separate handling of the one-parameter family.
\begin{Lemma}\label{th:quad2}
For $\lambda_0\neq 0$, and $\zeta_1, \zeta_2$ the critical points of $f_{\lambda_0, \lambda_\infty}$, we have
\[\min\left\{\hat{h}_{f_{\lambda_0, \lambda_\infty}}(\zeta_1), \hat{h}_{f_{\lambda_0, \lambda_\infty}}(\zeta_2) \right\}\geq \frac{1}{32}h(\lambda_\infty)-\frac{25}{32}h(\lambda_0)-\frac{47}{32}\log 2-\frac{3}{16}\log 3.\]
\end{Lemma}

Note that we may as well assume that $\lambda_\infty\neq 0$. In order to speak about the critical points, we introduce a variable $w$ satisfying
\[\lambda_0\lambda_\infty w^2+2w+1=0,\]
after which $f_{\lambda_0, \lambda_\infty}$ has critical points
\[\zeta_1=\lambda_0 w, \zeta_2=\frac{\lambda_0 w}{(2w+1)}\]
and branch points $\xi_i=-\zeta_i^2$, so
\[\xi_1=-\lambda_0^2w^2, \xi_2=\frac{-\lambda_0^2w^2}{(2w+1)^2}.\]
Note that it suffices to obtain a lower bound on $\hat{h}_{f_{\lambda_0, \lambda_\infty}}(\xi_1)$ in terms of $h(\lambda_0)$ and $h(\lambda_\infty)$, since the branch points are swapped by the involution $w\mapsto -w/(2w+1)$ which fix $\lambda_0$ and $\lambda_\infty$. We will assume, without loss of generality, that $\lambda_0$, $\lambda_\infty$, and $w$ are $K$-rational, since the constants we derive do not depend on $K$.

We first note the following lemma, quoted directly from~\cite{hcrit}.
\begin{Lemma}[{\cite[Lemma~20]{hcrit}}]\label{lem:quadgbound}
For any $v\in M_K$ and $z\in K_v$, 
\[g_{f_{\lambda_0, \lambda_\infty}, v}(z, \infty)\geq \log^+|z|_v-2\log\|1, \lambda_0, \lambda_\infty\|_v-\frac{3}{2}\log|1-\lambda_0\lambda_\infty|-\log^+|2|_v.\]
\end{Lemma}

Now, we define a set of places which depends on $\lambda_0$, $\lambda_\infty$, and $w$. Let \[C_v=\begin{cases}\log 2 & \text{if }v\text{ is archimedean or 2-adic}\\ 0 & \text{otherwise}.\end{cases}\]
and
\begin{equation}\label{eq:Sdefff}S=\left\{v\in M_K:\log |w|_v>\log^+|\lambda_0^{-1}|_v+C_v\right\}.\end{equation}

\begin{Lemma}
For each $k\geq 1$ and each $v\in S$, we have
\begin{multline*}g_{f_{\lambda_0, \lambda_\infty}, v}(f_{\lambda_0, \lambda_\infty}^k(\xi_1), \infty)\\\geq k(\log^+|\lambda_\infty^{-1}|_v-\epsilon_v) -2\log\|1, \lambda_0, \lambda_\infty\|_v-\frac{3}{2}\log|1-\lambda_0\lambda_\infty|-\log^+|2|_v.\end{multline*}
\end{Lemma}

\begin{proof}
By the Lemma~\ref{lem:quadgbound}, it is enough to show that, in the case $v\in S$, we have
\[\log^+|f_{\lambda_0, \lambda_\infty}^k(\xi_1)|_v\geq k(\log^+|\lambda_\infty^{-1}|_v-\epsilon_v).\]
Furthermore, as in the proof of \cite[Lemma~21]{hcrit}, this follows if we can show that 
\[\log|\xi_1|_v> \log\|\lambda_0, \lambda_\infty^{-1}\|_v+\log^+|2|_v.\]
Note that $v\in S$ already implies
\[\left|\frac{w}{2+1/w}\right|_v\leq |w|_v\cdot\begin{cases} 2 & \text{if }v\text{ is   2-adic}\\ 1 & \text{otherwise}.\end{cases}\]
So it follows from $v\in S$ that 
\begin{eqnarray*}
\log|\xi_1|_v&=&2\log|\lambda_0|_v+2\log^+|w|_v\\
&> &2\log|\lambda_0|_v+\log^+|\lambda_0^{-1}|_v+C_v+\log^+\left|\frac{w}{2+1/w}\right|_v\\&& - \begin{cases}\log 2 & \text{if }v\text{ is  2-adic}\\ 0 & \text{otherwise}.\end{cases}\\
&\geq &  \log^+|\lambda_0|_v+\log\left\|\lambda_0, \frac{\lambda_0 w}{2+1/w}\right\|_v+\log^+|2|_v\\
&\geq &\log\|\lambda_0, \lambda_\infty^{-1}\|_v+\log^+|2|_v,
\end{eqnarray*}
since $\lambda_0w/(2+1/w)=-\lambda_\infty^{-1}$.
\end{proof}

Drawing together what we have so far, note that we may use the previous lemma at places $v\in M_K$ and the trivial bound $\log^+|z|\geq 0$ to obtain for any $k\geq 1$
\begin{multline}\label{eq:quadk}
2^k\hat{h}_{f_{\lambda_0, \lambda_\infty}}(\xi_1)\geq \sum_{v\in S}\frac{[K_v:\QQ_v]}{[K:\QQ]}k\log^+|\lambda_\infty^{-1}|_v\\- k\log 12 - 2h(\lambda_0)-2h(\lambda_\infty)-\log 2.
\end{multline}

It remains to determine the  extent to which the height of $\lambda_\infty$ is supported by places in $S$.
\begin{Lemma}\label{lem:halfheight}
With $S$ defined as in~\eqref{eq:Sdefff}, we have
\[\sum_{v\in S}\frac{[K_v:\QQ_v]}{[K:\QQ]}\log^+|\lambda_\infty^{-1}|_v\geq \frac{1}{2}h(\lambda_\infty)-\frac{5}{2}h(\lambda_0)-\frac{7}{2}\log 2.\]
\end{Lemma}

\begin{proof}
We have (writing $n_v=[K_v:\QQ_v]/[K:\QQ]$)
\begin{eqnarray*}
\sum_{v\not\in S} n_v\log^+|\lambda_{\infty}^{-1}|_v&=&\sum_{v\not\in S} n_v\log^+\left|\frac{\lambda_0 w^2}{2w+1}\right|_v\\
&\leq &  \sum_{v\not\in S} n_v\log^+|w^2|_v+ \sum_{v\not\in S} n_v\log^+|\lambda_0|_v+ \sum_{v\not\in S} n_v\log^+\left|\frac{1}{2w+1}\right|_v\\
&\leq & \sum_{v\not\in S}2(n_v \log^+|\lambda^{-1}_0|+C_v)+h(\lambda_0)+h(2w+1)\\
&\leq & 3h(\lambda_0)+4\log 2 + h(2w+1).
\end{eqnarray*}
Now, since $\lambda_\infty = \lambda_0 w^2/(2w+1)$, one can check that
\[h(\lambda_\infty)\geq 2h(1+2w)-h(\lambda_0)-3\log 2\]
(treating $\lambda_\infty$ as a quadratic rational function in $2w+1$), and so we have
\[\sum_{v\not\in S} n_v\log^+|\lambda_{\infty}^{-1}|_v\leq \frac{1}{2}h(\lambda_\infty) +\frac{7}{2}h(\lambda_0)+\frac{11}{2}\log 2,\]
from which the claim in the lemma follows.
\end{proof}

Combining Lemma~\ref{lem:halfheight} with inequality~\eqref{eq:quadk}, we have
\begin{multline*}
2^{k+1} \hat{h}_{f_{\lambda_0, \lambda_\infty}}(\zeta_i)\geq \left(\frac{k}{2}-2\right)h(\lambda_\infty)-\left(\frac{7k}{2}+2\right)h(\lambda_0)\\-\left(\frac{15k}{2}+1\right)\log 2-k\log 3
\end{multline*}
for all $k\geq 0$, so taking $k=3$, we have
\[\hat{h}_{f_{\lambda_0, \lambda_\infty}}(\zeta_i)\geq\frac{1}{32}h(\lambda_\infty)-\frac{25}{32}h(\lambda_0)-\frac{47}{32}\log 2-\frac{3}{16}\log 3.\]
This completes the proof of Lemma~\ref{th:quad2}.

For example, if $\lambda_0$ is a root of unity and $f_{\lambda_0, \lambda_\infty}$ fails to have both critical orbits infinite, then
\[h(\lambda_\infty)\leq 47\log 2+6\log 3 \approx 39.17\]
Unforunately, enumerating all $\lambda_\infty$ up to this height and algebraic degree 3 (if $f_{\lambda_0, \lambda_\infty}$ is conjugate to a function defined over $\QQ$), presents computational challenges.

As noted, Theorem~\ref{th:quad} follows from Lemma~\ref{th:quad2}, except that the latter says nothing about quadratic morphisms of the form $z+a+z^{-1}$. The next lemma, then, completes the proof of Theorem~\ref{th:quad}.

\begin{Lemma}
Let $f_a(z)=z+a+z^{-1}$ have a fixed point of multiplier $\lambda$. Then the critical points $\zeta_1$, $\zeta_2$ satisfy
\[\hat{h}_f(\zeta_i)\geq Ah(a)-B.\]
\end{Lemma}

\begin{proof}
Note that the critical points of $f_a$ are $z=\pm 1$. It is straightforward to check that $\deg_a(f^n(\pm 1))=2^{n-1}$, and so on the generic fibre of the family, $\hat{h}_f(\pm 1)=\frac{1}{2}$.
By a result of Call and Silverman~\cite[Theorem~4.1]{call-silv}, we have (for any $\epsilon>0$)
\[\hat{h}_f(\pm 1)\geq \left(\frac{1}{2}-\epsilon\right)h(a)-C_\epsilon.\]
Taking $\epsilon<\frac{1}{2}$, and noting as above that $h_{\mathsf{M}_2}(f_a)\ll h(a)$, we have $\hat{h}_f(\pm 1)\gg h(f)$ in this family.
\end{proof}

\end{document}